\documentclass[11pt]{amsart}
\usepackage{amsmath,amstext,amsbsy, amssymb, amscd, amsxtra,amsthm, color, amsfonts}
\usepackage{graphicx, longtable}
\usepackage[enableskew,vcentermath]{youngtab}
\usepackage{caption, multirow}

\newtheorem{thm}{Theorem}[section]
\theoremstyle{plain}

\newtheorem{lem}[thm]{Lemma}
\newtheorem{prop}[thm]{Proposition}
\newtheorem{cor}[thm]{Corollary}

\theoremstyle{definition}

\theoremstyle{remark}
\newtheorem{remark}[thm]{Remark}

\numberwithin{equation}{section}

\newcommand{\C}{\mathbb C}
\newcommand{\Cl}{{\mathcal Cl}}

\newcommand{\tr}{\text{tr}}
\newcommand{\B}{\mathcal B} 

\newcommand{\typeM}{\texttt M}
\newcommand{\typeQ}{\texttt Q}

\newcommand{\be}{\beta}
\newcommand{\Z}{\mathbb Z}
\newcommand{\h}{V}

\newcommand{\mf}{\mathfrak}

\newcommand{\Hom}{\text{Hom} }

\newcommand{\wtd}{\widetilde}

\newcommand{\aHC}{\mathfrak{H}^{\mathfrak c}}

\title[Spin fake degrees for exceptional Weyl groups]
{Coinvariant algebras and fake degrees for spin Weyl groups of exceptional type}

\author[Baltera and Wang]{Constance Baltera and Weiqiang Wang}

\address{Department of Mathematics, University of Virginia, Charlottesville, VA 22904}
\email{cgb2k@virginia.edu (Baltera), \quad ww9c@virginia.edu (Wang)}

\begin{document}

\maketitle

\begin{abstract}
We compute all the spin
fake degrees for the exceptional Weyl groups, which are by definition the graded multiplicities of
the simple modules of a spin Weyl group in its spin coinvariant
algebra. The spin fake degrees are all shown to be palindromic
polynomials.
\end{abstract}

\let\thefootnote\relax

\section{Introduction}

 We formulated in  \cite{BW} (also see \cite{WW1}) the notion of spin coinvariant algebras and spin fake degrees
for every Weyl group $W$, 
which is a spin analogue of coinvariant algebras and fake degrees for Weyl groups \cite{Lu2}. 
The spin fake degrees
were computed in  \cite{WW1} for type $A$ and in \cite{BW} for the remaining classical types. 
The goal of this sequel to \cite{BW} is to compute all the spin fake degrees for the exceptional 
Weyl groups of types $G_2, F_4, E_6, E_7,$ and $E_8$.

The spin Weyl group algebra $\C W^-$ (associated to a distinguished double cover of $W$)
is naturally a superalgebra, and a $\C W^-$-module
is always understood to be $\Z_2$-graded  (see \cite[\S 2]{BW} for more detail). Informally speaking,
the spin fake degrees are  the graded multiplicities of
the simple modules of $\C W^-$  in its spin coinvariant algebra. 

We shall denote by $|\C W^-|$ the
underlying ungraded algebra. 
For $W$ of exceptional type, the split classes
of $W$ were classified (built on the work of Carter
\cite{Ca}), the simple ungraded $|\C W^-|$-characters
constructed, and the spin character tables  computed by Morris
\cite{Mo1}.

After a review in Section~\ref{sec:sfd} of the formulation 
of spin fake degrees in the framework of the ($\Z_2$-graded) module theory of the superalgebra $\C W^-$
 (see \cite[\S1, \S 3]{BW} for detail),
we establish in  Section~\ref{sec:Molien} a spin variant of Molien's formula,
a basic formula used in the computation of spin fake degrees in this paper.

In Section~\ref{sec:exceptional}, we classify the simple $\C
W^-$-modules and determine their types, refining the classification
of ungraded simple modules due to Morris \cite{Mo1}. Then we
write new code using CHEVIE \cite{CH} and \cite{GAP} to compute the spin fake
degrees in all exceptional cases, where as an input we use the spin character
tables computed in \cite{Mo1} (in which we note a typo in the $E_8$
spin character table). The GAP code for computing all spin fake degrees
for both exceptional and classical Weyl groups is available in \cite{Ba}. 
The spin fake degrees of all exceptional
types are tabulated (see Tables 1-5) in Section~\ref{sec:tables}.
We also take this opportunity to add Tables 6-9 for spin fake degrees 
of some classical Weyl groups of low rank.
Combining with the results for the classical types in \cite{BW},
we observe that all the spin fake degrees for all Weyl groups are 
palindromic. A similar palindromicity was observed for the usual
fake degrees by Beynon-Lusztig \cite{BL}.

\vspace{.3cm}

{\bf Acknowledgments.}
The second author is partially supported by
NSF grant DMS--1101268.
GAP and CHEVIE are indispensable to this paper.

\section{Spin coinvariant algebras and spin fake degrees}
\label{sec:sfd}

In this section we recall from \cite{BW} the notion of spin Weyl groups, spin coinvariant algebras,
and spin fake degrees.   

\subsection{Spin Weyl groups} \label{weyl gps}

Let $W$ be a 
Weyl group with the following presentation:
\begin{eqnarray} \label{eq:weyl}
\langle s_1,\ldots,s_n \mid  (s_is_j)^{m_{ij}} = 1,\ m_{i i} = 1,
 \ m_{i j} = m_{j i} \in \Z_{\geq 2}, \text{for }  i
 \neq j \rangle.
\end{eqnarray}
We shall label the vertices of the Coxeter diagrams of the exceptional Weyl groups as in \cite[\S 2.1]{BW}. 

In this paper (as in \cite{KW, BW}), we shall be concerned exclusively
with a distinguished double covering $\wtd{W}$ of $W$:
\begin{equation}
\label{eq:ses}
 1 \longrightarrow \Z_2 \longrightarrow \wtd{W}
\stackrel{\theta}{\longrightarrow} W \longrightarrow 1.
\end{equation}
We denote by $\Z_2 =\{1,z\}$.
%
The quotient algebra of $\C \wtd{W}$ by the ideal generated by $z+1$
is denoted by $\C W^-$ and called the {\em spin Weyl group algebra}
associated to $W$. 
The double cover $\wtd{W}$ is chosen so that
the spin Weyl group algebra $\C W^-$ has the following
uniform presentation: $\C W^-$ is the algebra generated by $t_i,
1\le i\le n$, subject to the relations
\begin{equation}
(t_{i}t_{j})^{m_{ij}} = (-1)^{m_{ij}+1}, \qquad \forall i,j.
\end{equation}
The algebra $\C W^-$ has a natural
superalgebra structure by letting each $t_i$ be odd.

\subsection{Spin coinvariant algebras and spin fake degrees}

Denote by $\h$ the irreducible reflection representation of the Weyl group $W$.
Note that $\h$ carries a $W$-invariant nondegenerate bilinear form
$(\cdot,\cdot)$, and let
$\Cl_\h$ be the Clifford algebra associated to $(\h, (\cdot,\cdot))$.
The action of $W$ on $\h$ preserves the bilinear form
$(\cdot,\cdot)$ and thus $W$ acts as automorphisms of the algebra
$\Cl_\h$. This gives rise to a semi-direct product
$
\aHC_W :=\Cl_\h \rtimes  W,
$
which is called the {\em Hecke-Clifford algebra} for $W$. Note that
$\Cl_\h$ is naturally an $\aHC_W$-module. 
The algebra $\aHC_W$ is endowed with a superalgebra
structure by letting each element in $W$ be even and each $\be_i$ be
odd.

A module over a superalgebra $ A$
is always understood in this paper as a $\Z_2$-graded $
A$-module $M = M_{\bar 0} \oplus M_{\bar 1}$ whose grading is
compatible with the action of $ A$; that is, $ A_iM_j
\subseteq M_{i+j}$. We shall denote by $| A|$ the underlying
algebra of $A$ with $\Z_2$-grading forgotten, and by $|M|$ the
$| A|$-module with the $\Z_2$-grading of $M$ forgotten.
We denote by $A\text{-}\mf{mod} $ the category of finite-dimensional $A$-modules.

Accordingly to \cite[Theorem 2.4]{KW},
there exists an explicit isomorphism of superalgebras:
$$
\Phi:  \aHC_W = \Cl_\h \rtimes  W
\stackrel{\simeq}{\longrightarrow} \Cl_\h \otimes \C W^-,
$$
which extends the identity map on $\Cl_\h$.
We say the superalgebras $\aHC_W$ and $\C W^-$ are Morita super-equivalent, as there exists a functor
$
\mathfrak{G} : 
B\text{-}\mf{mod} \longrightarrow A\text{-}\mf{mod},$
which is almost an equivalence of categories (see \cite[Proposition~3.3]{BW}).
In particular, it follows by \cite[Theorem~3.5]{BW} that we have $\mf G (\Cl_\h) \cong \B_W$ as $\C W^-$-modules, 
where $\B_W$ is the basic spin module of $\C W^-$  (see \cite[\S 3.3]{BW}).

The Weyl group $W$ acts on $V$
as its reflection representation, and then on the symmetric
algebra $S^*V$. The
coinvariant algebra $(S^*V)_W =S^*V/\langle(S^*V)^W_+\rangle$, where
$\langle(S^*V)^W_+\rangle$ denotes the ideal generated by the
homogeneous $W$-invariants of positive degrees, is a graded regular
representation of $W$.
Following \cite{BW}, we call $\Cl_V\otimes (S^*V)_W$
the {\em spin coinvariant algebra} for $W$.
Note that
$$
\Cl_V\otimes (S^*V)_W=\bigoplus_k \Cl_V \otimes
(S^k V)_W
$$
is a
graded regular representation of the Hecke-Clifford superalgebra
$\aHC_W$, where $\Cl_V$ acts by left multiplication on the first
tensor factor and $W$ acts diagonally.
The functor $\mf G$ sends the $\aHC_W$-module $\Cl_V\otimes (S^*V)_W$ to 
the $\C W^-$-module $\B_W \otimes (S^*V)_W$.

Let $\chi$ be a simple $\aHC_W$-module or its character, and
let $\chi^-$ be a simple $\C W^-$-module or its character 
corresponding to $\chi$ under the Morita super-equivalence
$\mf G$. Define
\begin{align}  \label{eq:sfdegree}
\begin{split}
P_W(\chi, t)  &= \displaystyle \sum_k \dim \Hom_{\aHC_W} (\chi,
\Cl_\h \otimes (S^kV)_W) t^k,
 \\
P_W^-(\chi^-, t) &= \displaystyle \sum_k \dim \Hom_{\C W^-} (\chi^-,
\B_W \otimes (S^kV)_W) t^k,
  \\
H_W^-(\chi^-, t) &= \displaystyle \sum_k \dim \Hom_{\C W^-} (\chi^-,
\B_W \otimes S^kV) t^k.
\end{split}
\end{align}  
The precise relations among $P_W(\chi, t)$,
$P_W^-(\chi^-, t)$, and $H_W^-(\chi^-, t)$ were determined in \cite[\S3]{BW}. 
By \cite[Lemma~3.11]{BW}, we have 
\begin{equation}  \label{eq:P=H}
P_W^-(\chi^-, t) =H_W^-(\chi^-, t) \prod_{i=1}^n(1-t^{d_i}),
\end{equation}
where the $d_i$'s are the degrees of the Weyl group $W$. 
Following \cite{BW}, we call the polynomial 
$P_W^-(\chi^-, t)$ the {\em spin
fake degree} of the simple $\C W^-$-character $\chi^-$.

The main goal of this paper is to compute the spin fake degrees
$P_W^-(\chi^-, t)$ for every exceptional Weyl group $W$ and every
simple $\C W^-$-character $\chi^-$.

\section{Spin Molien's formula}
\label{sec:Molien}

In this section we formulate a
spin version of Molien's formula for later use in computing the
spin fake degrees.
Recall the double cover
$\theta: \wtd{W} \rightarrow W$ from \eqref{eq:ses}.

\begin{lem}  \label{lem:zerotr}
Let $\psi$ be a simple character of $\C W^-$ (or, equivalently, of $\wtd W$
where $z$ acts as $-1$), and let $\tilde x\in \wtd W$.
Then $\psi(\tilde x) =0$, unless $\tilde x \in \wtd W$ is even split.
\end{lem}

\begin{proof}
Let $M$ be the ($\Z_2$-graded) module of $\wtd W$ underlying $\psi$.
For $\tilde x$ odd, $\tilde x$ switches the even and odd parts of $M$, and
hence the trace of $\tilde x$ on $M$ is zero, i.e., $\psi(\tilde x)=0$.

If $\tilde x$ is even and non-split, then $\tilde x$ is conjugate to
$z\tilde x$ by definition. So we have the trace identity on $M$:
$\tr (\tilde x) =\tr (z\tilde x) =-\tr(\tilde x)$, since $z$ acts by
$-1$ on $M$. Hence again $\psi(\tilde x) =\tr(\tilde x) =0$.
\end{proof}

We shall denote by $\langle x\rangle$ the conjugacy class of $x\in
W$, and $C_x$ the centralizer of $x$ in $W$. The following
proposition is a variation on Molien's formula. For a finite group
$G$, we denote by $G_*^{es}$ the set of even split conjugacy classes
of $G$; cf. \cite{Joz1} and \cite[\S 2]{BW}.

\begin{prop} [Spin Molien's formula]
 \label{general sfd}
Let $W$ be a Weyl group, and let $\chi^-$ be a simple $\C
W^-$-character.  Then the graded multiplicity of $\chi^-$ in the $\C
W^-$-module $\B_W \otimes S^*V$ is
$$
H_W^-(\chi^-, t) =
\sum_{\langle x\rangle\in W_*^{\text{es}}} \frac{\chi^-(\tilde x)
\tr(\tilde x)|_{\B_W}}{|C_x|
\det(1-tx)},
 $$
 where $\tilde x\in \wtd W$ is chosen such that $x=\theta(\tilde x)$.
\end{prop}

\begin{proof}
The character values of the $\wtd W$-modules $\B_W$ and $S^*V$ are all real.
Using the standard $\wtd
W$-character inner product $(\cdot,\cdot)$ (see \cite[Theorem~4.12]{Joz1}), we compute
\allowdisplaybreaks{
\begin{align*}
H_W^-(\chi^-, t) &=(\chi^-, \B_W\otimes S_tV)
\\
& = \big(\chi^-, \B_W\otimes \sum_{j\ge
0} t^j (S^jV) \big)
\\
&= \frac1{|\wtd W|}\sum_{\tilde x\in \wtd
W} \chi^-(\tilde x)\tr(\tilde x^{-1})|_{\B_W}
\Big ( \sum_{j\ge0}\tr (x^{-1})|_{S^jV} t^j\Big)
\\
& = \frac1{2|W|}\sum_{\tilde x\in\wtd
W}\frac{\chi^-(\tilde x)\tr(\tilde x)|_{\B_W}}{\det(1-tx)}.
\end{align*}}
We first reorganize this sum over conjugacy classes $\langle\wtd
x\rangle$ of $\wtd W$. By Lemma~\ref{lem:zerotr}, we may restrict
$\langle\tilde x\rangle$ to even split classes from now on.  The
split conjugacy classes in $\wtd W$ are obtained as half of the
inverse image of split conjugacy classes in $W$, that is,
$\langle\wtd x\rangle \sqcup z\langle \wtd x\rangle =
\theta^{-1}(\langle x\rangle).$  The classes $\langle \wtd x\rangle$
and $z\langle \wtd x\rangle$ are of the same size (as the class
$\langle x\rangle$), but have opposite character values.
We may reorganize the sum now over even
split conjugacy classes of $W$. The details are as follows:
\begin{align*}
\frac1{2|W|}\sum_{\tilde x\in\wtd
W}\frac{\chi^-(\tilde x)\tr(\tilde x)|_{\B_W}}{\det(1- tx)}
&=\frac 12 \sum_{\langle \tilde x \rangle \in \wtd W_*^{es}}
\frac{\chi^-(\tilde x)\tr(\tilde x)|_{\B_W}}{|C_x|\det(1- tx)}
 \\
& =
\sum_{\langle x\rangle \in W_*^{es}}
\frac{\chi^-(\tilde x)\tr(\tilde x)|_{\B_W}}{|C_x|\det(1- tx)}.
\end{align*}
The proposition is proved.
\end{proof}

We have the following corollary of Proposition~\ref{general sfd}
thanks to \eqref{eq:P=H}.

\begin{cor} \label{molien cor}
Let $W$ be a Weyl group with degrees $d_1,\ldots, d_n$, and let
$\chi^-$ be a simple $\C W^-$-character.  Then the spin fake
degree of $\chi^-$ is
$$
P_W^-(\chi^-, t) =
\sum_{\langle x \rangle\in W_*^{es}}
\frac{\chi^-(\tilde x)\tr(\tilde x)|_{\B_W}}{|C_x| \det(1-tx)} \cdot \prod_{i=1}^n
 (1-t^{d_i}),
 $$
 where $\tilde x\in \wtd W$ is chosen such that $x=\theta(\tilde x)$.
\end{cor}

\section{The spin fake degrees of exceptional Weyl groups}
\label{sec:exceptional}

Let $W$ be an exceptional Weyl group throughout this section.

\subsection{List of split classes}

We first briefly recall Carter's parametrization of conjugacy
classes of Weyl groups by admissible diagrams \cite{Ca} as follows.
Given a conjugacy class, choose a representative element $w$, and
decompose it into a product of two involutions subject to certain
conditions (see \cite[Section 3]{Ca}).  Each involution is a product
of reflections corresponding to mutually orthogonal roots, and the
admissible diagram is a graph whose nodes correspond to (the roots
associated to) the reflections, with the edge between nodes
corresponding to roots $r$ and $s$ having weight
$4\frac{(r,s)(s,r)}{(r,r)(s,s)}$.  Many of the possible graphs
resemble Dynkin diagrams and are named accordingly. Carter shows
that if $w$ has an admissible diagram $\Gamma$, then so must all its
conjugates \cite[p.6]{Ca}, and that we may describe the conjugacy
classes of any Weyl group $W$ by such admissible diagrams
\cite[p.45]{Ca}.

Morris \cite{Mo1} has determined the split conjugacy classes for
the exceptional Weyl groups $W$ (with the double covers $\wtd W$),
and their descriptions are given in terms of Carter's
parametrization by admissible diagrams.

We also need to determine the parity of the split classes, and
this can be done in a simple and precise manner. It turns out that
the nodes of the admissible diagram that labels a conjugacy class of
$W$ correspond to the reflections in a certain decomposition of an
element in that class. Since all reflections in $W$ are odd, the
parity of each split conjugacy class of $W$ may be read off from the
number of nodes in the corresponding admissible diagram. Below we
summarize Morris' classification of split classes, enhanced by the
parity separation.

\begin{prop}  \label{prop:cc:except}
A complete list of split classes of every exceptional Weyl group $W$
is as follows.
\begin{enumerate}
\item[($E_6$)]
There are $9$ even split classes: $\emptyset, A_2, A_4, 2A_2,
D_4(a_1), 3A_2, E_6,$ $E_6(a_1),$ $E_6(a_2)$.  There are $4$ odd
split classes: $D_5, D_5(a_1), D_3+D_2,$ and $A_4+A_1$.

\item[($E_7$)]
There are $13$ even split classes: $\emptyset, D_4(a_1)$, $A_2$,
$2A_2$, $3A_2$, $E_6(a_2)$, $E_6$, $E_6(a_1)$, $A_4$, $A_4+A_2$,
$D_6(a_1)$, $A_6,$ and $D_6$;

\noindent There are $13$ odd split classes: $7A_1, 2A_3+A_1$,
$D_4+3A_1$, $D_6(a_2)+A_1$, $E_7(a_4),$ $A_5+A_2$, $E_7(a_2)$,
$E_7$, $D_6+A_1$, $E_7(a_3), A_7,$ $E_7(a_1), A_4+A_1$.

\item[($GFE$)]
All $3, 9, 30$ split classes of $G_2$, $F_4,$ and $E_8$ listed in
\cite[\S8,\S9]{Mo1} are even.
\end{enumerate}
\end{prop}

\subsection{Classification of simple modules}

Recall a simple module of a superalgebra $A$ always has endomorphism algebra of dimension $1$ or $2$, cf. \cite{Joz1} and \cite[\S 2]{BW}.
A simple $A$-module
 is called type $\texttt M$ (and respectively, type $\texttt Q$)
if its endomorphism algebra is of dimension $1$ (and respectively, $2$).
Moreover, there is a simple recipe which determines the numbers of
simple $\C W^-$-modules of type $\texttt M$ and of type $\texttt Q$ by knowing
the numbers of even split and odd split classes of $W$, cf. \cite[Proposition~4.14]{Joz1} and \cite[Proposition~2.5]{BW}.

Now Proposition~\ref{prop:cc:except}
allow us to determine the numbers of simple $\C W^-$-modules of type
$\typeM$ and type $\typeQ$. Morris \cite{Mo1} (also cf. Read
\cite{Re1} for $F_4$) has determined the spin character table of all
the {\em ungraded} simple characters of $|\C W^-|$. It turns out
that we may determine which pairs of ungraded simple characters add
to become type $\typeQ$ characters of $\C W^-$ since the character
values for any $\C W^-$-module on the odd split classes must be 0.

We will follow Morris's notation \cite{Mo1}, referring to the
irreducible spin characters for the exceptional Weyl groups by their
degrees with a subscript $s$ (or $ss$, $sss$, $ssss$, for multiple
characters of the same degree).  In cases where two ungraded simple
characters of the same degree add to become a type $\typeQ$
character (which only happens in $E_6, E_7$), we will denote the type
$\typeQ$ character by its degree with the shortest possible
subscript and a superscript $\typeQ$.

The Weyl group $E_7$ is the direct product of the Chevalley group
$B_3(2)$ and a cyclic group $\langle\xi\rangle$ of order $2$, where
all elements in $B_3(2)$ are even while the generator $\xi$ is odd.
Hence the simple $E_7$-modules are exactly the tensor products of
simple $B_3(2)$-modules with the unique two-dimensional simple
module of $\langle\xi\rangle$ (of type $\typeQ$), and they are
manifestly all of type $\typeQ$. Table~E summarizes the conversion
between the new type $\typeQ$ notation and Morris's ungraded simple
characters.

\begin{center}
\vspace{.2cm}
{Table E: Type $\typeQ$ characters as sums of ungraded simple characters}
\vspace{.2cm}

\resizebox{5.6in}{!}{\begin{tabular}[h]{|c| c| ccccc|}
\hline
\multirow{2}{*}{$E_6$}
& Type $\typeQ$ character &$120_{ss}^\typeQ$ &$160_s^\typeQ$
&$40_{sss}^\typeQ$ &$128_s^\typeQ$ & \\
& Ungraded characters &$60_s + 60_{ss}$ &$80_s+80_{ss}$
&$20_s+20_{ss}$ &$64_s+64_{ss}$ & \\
\hline\hline
\multirow{6}{*}{$E_7$}
& Type $\typeQ$ character &$16_s^\typeQ$
 &$96_s^\typeQ$ &$336_s^\typeQ$ &$560_s^\typeQ$ &$224_s^\typeQ$  \\
& Ungraded characters &$8_s+8_{ss}$ &$48_s+48_{ss}$
&$168_s+168_{ss}$ &$280_s+280_{ss}$ &$112_s+112_{sss}$ \\
\cline{2-7}
& Type $\typeQ$ character &$224_{ss}^\typeQ$ &$1024_s^\typeQ$
 &$1440_s^\typeQ$ &$1120_s^\typeQ$ &$896_s^\typeQ$ \\
& Ungraded characters &$112_{ss}+112_{ssss}$
 & $512_s+512_{ss}$& $720_s+720_{ss}$&$560_s+560_{ss}$
&$448_s+448_{ss}$ \\ \cline{2-7}
& Type $\typeQ$ character & $240_s^\typeQ$ &$128_s^\typeQ$ &$128_{ss}^\typeQ$& & \\
& Ungraded characters &$120_s+120_{ss} $& $64_s+64_{sss}$&$64_{ss}+64_{ssss}$ & & \\
\hline
\end{tabular}}
\vspace{.2cm}
\end{center}

In this way, we have upgraded the results of Morris \cite{Mo1} into the
following.

\begin{prop}
The types of the simple $\C W^-$ modules, for $W$ an exceptional
Weyl group, are as follows.
\begin{enumerate}
\item
$\C E_6^-$ has $5$ type $\typeM$ simple modules $8_s, 40_s, 72_s,
40_{ss}, 120_s$, and $4$ type $\typeQ$ simple modules
$120_{ss}^\typeQ, 160_s^\typeQ, 40_{sss}^\typeQ, 128_s^\typeQ$.

\item All $13$ simple $\C E_7^-$-modules are of type $\typeQ$.

\item
All $30$ simple $\C E_8^-$-modules are of type $\typeM$, as are all
$9$ simple $\C F_4^-$-modules and all $3$ simple $\C G_2^-$-modules.
\end{enumerate}
\end{prop}

\subsection{Spin fake degrees}
\label{subsec:sfdegexcept}

The main computational tool for the spin fake degrees of the
exceptional Weyl groups is the spin Molien's formula, see
Proposition~\ref{general sfd} or Corollary~ \ref{molien cor}. We
 implement this using CHEVIE; code is available upon request. 
To that end, as inputs we shall need the
spin character tables of $\C W^-$, which were computed by Morris
\cite{Mo1} in the ungraded setting.

More precisely, we use the spin character tables of Morris for $E_6$
\cite[Table III]{Mo1}, $E_7$ \cite[Table IV]{Mo1}, $E_8$ \cite[Table
V]{Mo1}, and $G_2$ \cite[Table VI]{Mo1}. In cases of $E_6$ and
$E_7$, we need to add suitable pairs of columns in the spin
character tables to form the characters of type $\typeQ$ simple
modules.

\begin{remark}
There is a typo in the $E_8$ spin character table in \cite{Mo1}: the thirteenth
entry of the last character should be 2 rather than $-2$. This is
detected and corrected using the orthogonality relations of simple
characters.
\end{remark}

For $W$ of type $F_4$, Read has labeled its 9 simple characters as
$\phi_1,\ldots,\phi_9$. See Table~ F for a comparison between Read's
notation \cite[Table 1]{Re1} and Morris's \cite[Table VII]{Mo1}.

\begin{center}
\vspace{.2cm}
{Table F: Two labelings of the spin character table for $F_4$}
\vspace{.2cm}

\begin{tabular}[h]{|c| c c c c c c c c c|}
\hline Morris's labels & $4_s$ & $4_{ss}$ & $8_{sss}$ & $8_{ssss}$
 & $12_{ss}$ & $12_s$ & $8_s$ & $24_s$ & $8_{ss}$\\
\hline Read's labels & $\phi_1$ & $\phi_2$ & $\phi_3 $& $\phi_4 $&
 $\phi_5 $& $\phi_6 $& $\phi_7$ & $\phi_8 $& $\phi_9$\\
\hline
\end{tabular}
\vspace{.2cm}
\end{center}
The character values which we use in the computation of the spin
fake degrees for $F_4$ are those given by Read \cite[Table~1]{Re1}.
They differ by a sign  in the conjugacy classes $A_2, \wtd A_2$, and
$B_2$ from those given by Morris \cite[Table~VII]{Mo1} because of a
different choice of these $\wtd F_4$-conjugacy classes (differing by
a factor $z=-1$). However this does not affect the computation of
the spin fake degrees.

We summarize our CHEVIE computations in
Proposition~\ref{exceptionpalindrome} and Theorem~\ref{th:sfd:excep}
below. Denote by $N$ the number of reflections in $W$. The following can be 
observed by inspection case-by-case from our CHEVIE computation.

\begin{prop}\label{exceptionpalindrome}
Let $W$ be an arbitrary exceptional Weyl group. Then for every
simple $\C W^-$-character $\chi$, we have
$$
P_W^-(\chi, t) = t^N P_W^-(\chi, t^{-1}).
$$
\end{prop}
(The statement already holds for classical Weyl groups \cite{BW}.)


We exploit this property when presenting the spin fake degrees for
the exceptional groups in Table~ \ref{sfd G2} through Table~\ref{sfd
E8} in Section~\ref{sec:tables}.  In each column of the tables, we
list only the coefficients of the spin fake degrees $P_W^-(\chi, t)$
(which are polynomials in $t$) for degrees $0$ through $\frac N 2$.
The remaining halves of coefficients can be determined via
palindromicity in Proposition~\ref{exceptionpalindrome}.

\begin{thm}  \label{th:sfd:excep}
For $W$ exceptional, the coefficients of the spin fake degrees
$P_W^-(\chi, t)$ are given in Table~\ref{sfd G2} through
Table~\ref{sfd E8} in Section~\ref{sec:tables}.
\end{thm}

\begin{proof}
These values are computed using Corollary~ \ref{molien cor}.  The
computations are performed using the GAP~3 package CHEVIE \cite{GAP,
CH}. 
\end{proof}

\section{Tables for spin fake degrees} 
\label{sec:tables}

We use the notation and convention as specified in
Section~\ref{sec:exceptional}. In particular, the characters with
superscripts $\typeQ$ in Table~\ref{sfd G2} through Table~\ref{sfd
E8} are of type $\typeQ$, and those without are of type $\typeM$.
Note that in all tables, the character of the basic spin module
$\B_W$ is listed first.

In Table 6 through Table 9, we also list the spin fake degrees
of classical Weyl groups of type $B_2, A_4, B_4,$ and $D_4$ (see \cite{BW} for notation).

\begin{table}[htbp]
\parbox{.43\linewidth}{\centering
\caption{Spin fake degrees for type $G_2$} \label{sfd G2}
\resizebox{1.2in}{!}{\begin{tabular} {|l|*{3}{c}|} \hline
& $2_s$ & $2_{ss}$ & $2_{sss}$\\
\hline
0&1&&\\
1&1&&1\\
2&0&1&1\\
3&0&2&0\\
\hline
\end{tabular} }}
\hfill
\parbox{.56\linewidth}{\centering
\caption{Spin fake degrees for type $F_4$}\label{sfd F4}
\resizebox{2.5in}{!}{\begin{tabular} {|l|*{9}{c}|} \hline
&$4_s$ & $4_{ss}$ & $8_{sss}$ & $8_{ssss}$ & $12_{ss}$ & $12_s$ & $8_s$ & $24_s$ & $8_{ss}$\\
\hline
0&1&&&&&&&&\\
1&1&&&&&1&&&\\
2&0&&&&&1&&1&\\
3&0&&1&1&&0&&2&\\
4&0&&2&2&&1&&2&1\\
5&1&&1&1&1&2&1&3&1\\
6&1&1&0&0&3&3&2&4&0\\
7&1&2&1&1&4&3&2&4&2\\
8&1&1&3&3&3&2&3&5&3\\
9&0&0&3&3&2&2&2&8&2\\
10&0&0&2&2&4&3&0&9&2\\
11&1&2&2&2&5&4&3&7&3\\
12&2&4&2&2&4&4&6&6&4\\
\hline
\end{tabular}}}
\end{table}
\begin{table}[!h]
\caption{Spin fake degrees for type $E_6$}\label{sfd E6}
\begin{tabular} {|l|*{9}{c}|}
\hline
& $8_s$&$40_s$&$72_s$&$40_{ss}$&$120_s$&$120_{ss}^\typeQ$&$160_s^\typeQ$&$40_{sss}^\typeQ$&$128_s^\typeQ$\\
\hline
0&1&&&&&&&&\\
1&1&1&&&&&&&\\
2&0&1&&&&2&&&\\
3&0&1&&&1&4&&&\\
4&1&3&1&&2&4&2&&\\
5&2&4&3&&3&6&4&&2\\
6&1&4&4&&7&10&6&&4\\
7&1&5&5&1&11&14&12&&6\\
8&2&8&9&3&13&18&18&4&12\\
9&2&9&13&6&18&24&24&8&18\\
10&1&9&14&8&27&30&34&6&24\\
11&2&11&19&8&34&34&44&8&34\\
12&4&14&27&13&39&38&52&14&44\\
13&3&15&29&19&47&44&64&16&50\\
14&1&14&30&18&55&52&74&18&60\\
15&2&16&35&20&59&56&80&24&70\\
16&4&19&39&28&62&56&88&26&72\\
17&3&18&40&26&67&58&92&24&76\\
18&2&16&40&20&70&60&92&24&80\\
\hline
\end{tabular}
\end{table}
\clearpage

\clearpage
\begin{table}[h]
\caption{Spin fake degrees for type $E_7$}\label{sfd E7}
\noindent\resizebox{5in}{!}{\begin{tabular} {|l|*{13}{c}|} \hline
&$16_s^\typeQ$&$96_s^\typeQ$&$336_s^\typeQ$&$560_s^\typeQ$&$224_s^\typeQ$&$224_{ss}^\typeQ$&$1024_s^\typeQ
$&$1440_s^\typeQ$&$1120_s^\typeQ$&$896_s^\typeQ$&$240_s^\typeQ$&$128_s^\typeQ$&$128_{ss}^\typeQ$\\
\hline
0&2&&&&&&&&&&&&\\
1&2&2&&&&&&&&&&&\\
2&0&2&2&&&&&&&&&&\\
3&0&0&4&2&&&&&&&&&\\
4&0&2&4&4&&&2&&&&&&\\
5&2&4&6&4&2&&4&2&&&&&\\
6&2&6&8&6&4&&6&6&&&2&&\\
7&2&6&10&10&4&&12&10&&2&4&&\\
8&2&6&14&16&6&&18&16&4&4&2&&\\
9&2&8&20&22&8&2&26&24&12&6&2&&\\
10&2&10&26&28&10&4&38&36&18&14&6&&\\
11&2&12&30&36&14&4&52&56&26&22&10&2&2\\
12&4&14&36&46&20&8&68&80&40&32&14&4&4\\
13&4&18&44&58&26&14&90&104&60&50&18&4&4\\
14&4&18&54&72&30&18&116&136&86&68&22&8&8\\
15&2&18&64&90&34&22&142&178&116&90&28&12&12\\
16&4&22&72&108&40&30&176&222&150&122&36&14&14\\
17&6&28&82&124&52&40&212&274&192&154&44&20&20\\
18&6&30&94&142&62&48&250&334&240&188&56&26&26\\
19&4&30&104&166&66&58&294&396&288&234&68&30&30\\
20&4&32&116&190&74&68&338&462&346&278&76&38&38\\
21&6&36&130&212&86&82&382&530&410&322&86&48&48\\
22&6&40&140&234&94&96&430&600&468&376&100&52&52\\
23&6&40&148&256&102&104&476&676&526&424&114&60&60\\
24&6&42&158&278&112&116&518&748&588&468&124&70&70\\
25&8&46&170&298&120&132&562&806&648&520&134&74&74\\
26&6&48&180&316&126&142&600&866&702&562&144&82&82\\
27&6&46&186&332&130&146&632&926&748&594&154&90&90\\
28&6&48&190&346&136&156&662&968&784&632&164&92&92\\
29&8&52&196&354&144&166&684&1000&818&656&168&96&96\\
30&8&52&200&362&148&168&696&1026&842&668&170&102&102\\
31&6&50&200&368&142&168&706&1038&848&682&174&100&100\\
\hline
\end{tabular}}
\end{table}
\clearpage

\tiny \setlength\tabcolsep{.5pt}
\begin{longtable}{|l|*{15}{c@{\hspace{.2cm}} }|}
\caption{Spin fake degrees for type $E_8$ (Part~1)}\label{sfd E8}\\
\endfirsthead \caption[]{Spin fake degrees for type $E_8$ (Part 2)} \\
\endhead \hline
 &$16_s$&$112_s$&$320_s$&$448_s$&$224_s$&$448_{ss}$&$1680_s$&$2592_s$
 &$1344_s$&$5600_s$&$4800_s$&$2016_s$&$5600_{ss}$&$9072_s$&$800_s$\\
\hline
0&1&&&&&&&&&&&&&&\\[1pt]
1&1&1&&&&&&&&&&&&&\\[1pt]
2&0&1&&&&1&&&&&&&&&\\[1pt]
3&0&0&&&&1&&&&&&&&&\\[1pt]
4&0&0&&&&0&&&&&&1&&&\\[1pt]
5&0&0&&&&1&&&&&&1&1&&\\[1pt]
6&0&1&&&&2&1&&&&&0&2&&\\[1pt]
7&1&2&1&&&2&2&&&&&1&2&&1\\[1pt]
8&1&2&2&&&2&2&&&&&2&3&&2\\[1pt]
9&0&1&1&&&3&2&&&&&3&5&1&1\\[1pt]
10&0&1&0&&&3&2&&&&&5&8&2&0\\[1pt]
11&1&2&1&&&3&3&&&&&7&11&3&1\\[1pt]
12&1&3&3&&&5&6&1&&1&&7&13&6&3\\[1pt]
13&1&4&4&1&&7&10&3&&4&1&7&17&8&4\\[1pt]
14&1&4&4&2&&8&11&5&&6&3&10&24&10&4\\[1pt]
15&0&2&3&1&&8&11&5&&6&3&15&31&18&4\\[1pt]
16&0&1&2&0&&8&13&4&&7&3&21&39&29&5\\[1pt]
17&1&4&4&1&&10&16&8&&11&8&25&50&35&6\\[1pt]
18&2&7&8&4&&14&23&14&2&21&14&25&61&42&8\\[1pt]
19&2&8&11&6&1&16&32&19&5&32&17&28&71&57&13\\[1pt]
20&2&7&12&6&2&16&35&24&6&38&22&38&86&75&16\\[1pt]
21&1&5&9&6&1&17&35&26&7&45&32&49&109&97&14\\[1pt]
22&0&4&6&5&0&20&40&30&7&56&44&59&133&126&13\\[1pt]
23&1&7&11&6&1&23&50&41&9&71&56&67&154&156&20\\[1pt]
24&3&11&20&12&4&27&65&55&18&98&69&71&175&186&30\\[1pt]
25&2&12&23&18&5&32&80&71&26&132&89&76&203&222&34\\[1pt]
26&1&10&20&19&4&34&86&85&28&155&116&92&241&267&34\\[1pt]
27&1&8&18&16&5&33&86&90&32&171&137&117&281&328&36\\[1pt]
28&1&8&18&15&6&35&95&98&39&197&159&138&320&398&41\\[1pt]
29&2&12&24&21&7&42&115&125&46&243&200&146&364&456&49\\[1pt]
30&4&18&36&34&10&49&139&159&63&308&244&150&408&512&60\\[1pt]
31&4&19&42&43&16&52&161&183&84&369&277&165&452&593&72\\[1pt]
32&2&14&38&39&19&53&170&203&90&407&319&194&510&694&78\\[1pt]
33&1&11&31&35&14&55&171&220&93&445&377&228&580&799&75\\[1pt]
34&1&13&31&39&12&60&185&241&107&506&436&253&646&906&78\\[1pt]
35&2&17&43&48&20&68&217&284&126&584&493&265&702&1013&98\\[1pt]
36&4&23&59&64&29&74&253&335&157&682&552&274&758&1118&120\\[1pt]
37&5&25&64&78&34&78&277&373&187&782&622&294&829&1231&128\\[1pt]
38&2&20&55&75&34&81&283&401&193&843&705&332&915&1369&127\\[1pt]
39&1&15&48&65&31&81&285&421&200&889&776&379&999&1531&131\\[1pt]
40&3&18&54&70&35&83&305&448&227&973&842&411&1073&1684&145\\[1pt]
41&4&25&68&90&43&94&346&511&257&1094&936&417&1145&1804&163\\[1pt]
42&4&30&82&111&50&105&389&584&291&1229&1031&419&1218&1924&182\\[1pt]
43&5&30&87&120&58&106&413&622&328&1336&1100&449&1292&2080&197\\[1pt]
44&3&24&78&110&61&104&413&641&337&1387&1180&501&1380&2257&200\\[1pt]
45&0&18&65&99&52&108&413&667&336&1442&1281&545&1477&2426&196\\[1pt]
46&1&22&69&111&48&114&437&706&367&1551&1370&567&1556&2571&204\\[1pt]
47&5&31&92&135&69&120&482&770&410&1681&1447&571&1613&2692&233\\[1pt]
48&6&35&110&154&87&127&525&838&448&1807&1523&574&1670&2814&261\\[1pt]
49&5&34&107&161&83&130&540&873&476&1906&1603&600&1745&2951&263\\[1pt]
50&4&29&92&150&78&128&527&881&473&1940&1692&649&1832&3098&252\\[1pt]
51&2&23&82&135&74&127&522&891&471&1962&1760&693&1903&3251&254\\[1pt]
52&2&25&90&143&75&131&549&925&504&2052&1807&706&1948&3372&271\\[1pt]
53&5&35&111&173&92&140&592&989&542&2179&1878&690&1987&3428&291\\[1pt]
54&6&40&124&193&104&148&624&1047&566&2278&1947&680&2027&3482&304\\[1pt]
55&4&35&118&186&99&147&627&1053&580&2316&1974&706&2066&3589&308\\[1pt]
56&3&28&102&165&93&140&603&1030&569&2290&2004&753&2112&3694&301\\[1pt]
57&2&25&91&155&87&139&587&1028&553&2280&2054&779&2155&3753&289\\[1pt]
58&2&28&97&169&84&145&609&1055&575&2347&2078&767&2169&3775&294\\[1pt]
59&5&37&121&193&104&149&645&1095&608&2420&2079&741&2156&3766&319\\[1pt]
60&8&42&136&204&122&150&660&1116&618&2442&2080&730&2146&3754&334\\
\hline
\hline
&$2800_{ss}$&$5600_{sss}$&$7168_s$&$1120_s$&$8400_s$&$11200_s$&$6720_s$&$2800_s$
&$1344_{ss}$&$6480_s$&$8192_s$&$2016_{ss}$&$2016_{sss}$&$7168_{ss}$&$896_s$\\
\hline
0&&&&&&&&&&&&&&&\\
1&&&&&&&&&&&&&&&\\
2&&&&&&&&&&&&&&&\\[1pt]
3&&&&&&&&&1&&&&&&\\[1pt]
4&&&&&&&&&2&&&&&&\\[1pt]
5&&&&&&&&&2&&&&&&\\[1pt]
6&&1&&&&&&&2&&&&&&\\[1pt]
7&&2&&&&&1&1&2&&&&&&\\[1pt]
8&&2&1&&&&2&2&3&1&&&&&\\[1pt]
9&&3&3&&&1&2&1&5&1&&&&&\\[1pt]
10&&4&5&&1&3&3&0&7&0&&&&&\\[1pt]
11&1&5&7&&2&4&5&2&8&2&1&&&&\\[1pt]
12&2&9&8&&2&5&8&6&9&5&2&&&1&\\[1pt]
13&1&15&9&&3&9&12&8&10&7&3&&&2&\\[1pt]
14&1&19&14&&5&15&16&9&12&9&6&&&3&\\[1pt]
15&4&21&23&&10&20&20&10&18&13&9&&&6&\\[1pt]
16&7&24&32&&18&27&26&11&24&18&13&&&9&1\\[1pt]
17&8&33&41&&24&39&35&15&25&23&20&1&1&13&1\\[1pt]
18&9&49&49&1&28&53&47&23&25&31&28&3&3&20&0\\[1pt]
19&12&64&57&4&36&68&62&34&29&45&37&6&6&28&2\\[1pt]
20&18&74&75&7&51&87&77&42&36&60&51&10&9&37&4\\[1pt]
21&25&84&103&6&72&114&91&41&45&68&68&10&10&51&4\\[1pt]
22&31&99&129&3&98&147&109&41&54&79&86&10&12&68&5\\[1pt]
23&40&121&151&8&122&180&135&58&59&106&111&17&17&86&9\\[1pt]
24&50&154&173&18&142&217&168&84&61&140&140&26&26&111&12\\[1pt]
25&56&192&198&21&169&269&203&101&67&168&171&35&35&140&14\\[1pt]
26&67&219&240&22&212&331&235&109&79&196&212&43&40&170&18\\[1pt]
27&91&237&302&27&271&390&271&116&95&229&258&49&49&210&22\\[1pt]
28&115&266&360&32&335&457&316&129&109&268&306&58&63&256&29\\[1pt]
29&128&318&403&39&388&544&368&156&114&317&366&74&74&302&35\\[1pt]
30&141&386&446&52&434&638&430&195&115&375&432&94&91&360&38\\[1pt]
31&166&445&501&68&499&732&497&232&126&442&501&116&116&425&48\\[1pt]
32&200&483&579&80&594&838&557&252&148&508&583&136&131&490&61\\[1pt]
33&235&521&681&79&703&962&616&255&169&557&672&145&145&568&66\\[1pt]
34&266&579&775&80&810&1095&690&270&183&614&762&159&169&655&73\\[1pt]
35&298&659&840&106&902&1228&778&326&191&713&866&196&196&739&90\\[1pt]
36&333&754&902&138&984&1367&874&397&196&824&978&236&230&835&104\\[1pt]
37&364&841&991&153&1090&1528&970&435&207&907&1088&265&265&943&112\\[1pt]
38&405&897&1113&159&1236&1702&1052&445&232&981&1212&291&286&1045&126\\[1pt]
39&467&939&1252&169&1401&1863&1135&459&261&1069&1343&311&311&1158&142\\[1pt]
40&522&1011&1371&187&1548&2027&1239&496&277&1168&1469&338&353&1283&158\\[1pt]
41&550&1128&1448&214&1658&2219&1351&562&279&1283&1608&389&389&1397&175\\[1pt]
42&579&1253&1516&244&1758&2413&1464&637&282&1407&1752&438&425&1520&186\\[1pt]
43&634&1339&1625&270&1898&2589&1576&687&299&1521&1887&473&473&1656&203\\[1pt]
44&700&1383&1776&286&2085&2771&1669&700&330&1617&2032&507&504&1777&227\\[1pt]
45&756&1431&1929&284&2270&2968&1752&698&358&1695&2179&524&524&1902&240\\[1pt]
46&799&1521&2042&291&2413&3159&1858&731&368&1786&2314&547&566&2038&249\\[1pt]
47&837&1641&2106&341&2515&3335&1980&822&366&1926&2455&614&614&2156&273\\[1pt]
48&876&1755&2163&390&2610&3505&2092&912&369&2068&2596&674&655&2273&294\\[1pt]
49&917&1831&2269&393&2744&3681&2185&934&384&2147&2719&693&693&2402&300\\[1pt]
50&967&1858&2421&388&2921&3855&2255&913&410&2197&2844&710&713&2508&314\\[1pt]
51&1028&1878&2554&397&3087&3998&2320&915&437&2274&2967&731&731&2605&334\\[1pt]
52&1071&1949&2621&415&3187&4123&2406&963&445&2369&3067&753&772&2714&346\\[1pt]
53&1076&2064&2639&450&3228&4264&2497&1038&432&2464&3166&806&806&2798&356\\[1pt]
54&1083&2157&2661&479&3272&4392&2569&1097&426&2547&3260&853&827&2868&364\\[1pt]
55&1126&2177&2734&481&3369&4475&2619&1105&444&2599&3327&853&853&2949&372\\[1pt]
56&1174&2149&2849&475&3499&4546&2644&1072&470&2620&3391&852&861&3002&386\\[1pt]
57&1194&2146&2929&466&3586&4624&2661&1045&483&2626&3448&858&858&3036&392\\[1pt]
58&1192&2198&2925&465&3592&4677&2696&1070&477&2654&3476&860&876&3081&387\\[1pt]
59&1186&2265&2877&504&3553&4696&2738&1144&460&2718&3498&898&898&3100&394\\[1pt]
60&1184&2294&2852&536&3530&4700&2756&1188&450&2756&3512&932&904&3096&404\\
\hline
\end{longtable}
\clearpage \normalsize


\clearpage

\begin{table}[h]
\caption{Spin fake degrees for type $B_2$}
\begin{tabular} {|l|*{2}{c}|}
\hline
&$(2)$& $(1,1)$\\
\hline
0&1&\\
1&1&1\\
2&0&2\\
\hline
\end{tabular}
\end{table}

\begin{table}[h]
\caption{Spin fake degrees for type $A_4$}
\begin{tabular} {|l|*{3}{c}|}
\hline
&$(5)^\typeQ$& $(4,1)$& $(3,2)$\\
\hline
0&1&&\\
1&1&2&\\
2&1&4&2\\
3&2&6&4\\
4&2&8&6\\
5&2&8&8\\
\hline
\end{tabular}
\end{table}

\begin{table}[h]
\caption{Spin fake degrees for type $B_4$}
\begin{tabular} {|l|*{5}{c}|}
\hline
& $(4)$&$(3,1)$&$(2,2)$&$(2,1,1)$&$(1,1,1,1)$\\
\hline
0&1&&&&\\
1&1&1&&&\\
2&0&2&&1&\\
3&1&2&1&2&1\\
4&1&3&3&2&2\\
5&1&4&3&4&1\\
6&1&5&2&6&1\\
7&1&5&4&6&2\\
8&2&4&6&6&2\\
\hline
\end{tabular}
\end{table}

\begin{table}[h]
\caption{Spin fake degrees for type $D_4$}
\begin{tabular} {|l|*{4}{c}|}
\hline
&$\{(1,1,1,1), (4)\}$&$\{(2,1,1),(3,1)\}$&$\{(2,2)\}_+$&$\{(2,2)\}_-$\\
\hline
0&1&&&\\
1&1&1&&\\
2&0&3&&\\
3&2&4&1&1\\
4&2&5&3&3\\
5&1&7&3&3\\
6&2&8&2&2\\
\hline
\end{tabular}
\end{table}
\clearpage

\end{document}